\documentclass{article}
\usepackage[utf8]{inputenc}
\usepackage{amsmath,amssymb,amsthm}
\usepackage{bm}
\usepackage{algorithm,algorithmic}
\usepackage{enumerate}
\usepackage[margin=25mm]{geometry}
\usepackage[numbers]{natbib}
\usepackage{mathtools}
\mathtoolsset{showonlyrefs}
\usepackage[colorlinks=true,linkcolor=blue,urlcolor=black]{hyperref}

\newcommand{\argmin}{\mathop{\rm argmin}\limits}
\newcommand{\dom}{\mathop{\rm dom}}
\newcommand{\RR}{\mathbb{R}}
\newcommand{\NN}{\mathbb{N}}

\newtheorem{theorem}{Theorem}[section]
\newtheorem{lemma}{Lemma}[section]
\newtheorem{definition}{Definition}[section]
\newtheorem{assumption}{Assumption}[section]
\counterwithin{figure}{section}
\counterwithin{table}{section}

\title{
A proximal gradient method with adaptive backtracking for weakly smooth multiobjective optimization
}
\author{Yuki Miyazaki\thanks{Department of Mathematics, College of Science and Technology, Nihon University, Japan. E-mail: csyi25012@g.nihon-u.ac.jp}
\and
Masaru Ito\thanks{Department of Mathematics, College of Science and Technology, Nihon University, Japan. E-mail: ito.masaru@nihon-u.ac.jp}
\and
Shotaro Yagishita\thanks{Risk Analysis Research Center, The Institute of Statistical Mathematics, Japan, E-mail: syagi@ism.ac.jp} \thanks{Center for Social Data Structuring, Joint Support-Center for Data Science Research, Japan}}
\date{\today}
\begin{document}
\maketitle
\begin{abstract}
In this paper, we propose a proximal gradient method with adaptive linesearch for multiobjective optimization problems whose objective functions are weakly smooth, i.e., they have H\"older continuous gradients. The proposed method is parameter-free as we do not require prior knowledge of parameters related to the weak smoothness of the objective function; the proposed linesearch finds an appropriate step-size that adapts to the weak smoothness.
The complexity guarantee analyzed in this paper for the non-convex case is compatible with related works and our algorithm accepts coercer stationarity measure compared to existing methods. 
We also establish a novel complexity result for the convex case which improves the one in non-convex case.

\end{abstract}

\section{Introduction}\label{sec:intro}
Multiobjective optimization is a problem that involves multiple objective functions to be optimized simultaneously. A major optimality concept for this problem is the Pareto optimality identifying points on which any single objective function cannot be improved without degrading other objective functions.
Multiobjective optimization has wide applications in problems involving multi-decisions such as  economics, finance, robust optimization, and many others (see, e.g., \cite{evans1984overview, fukuda2014survey, marler2004survey, sener2018multi}).

In this paper, we consider a multiobjective optimization problem 
\begin{equation}\label{prob}
\begin{array}{ll}
\text{minimize} & F(x) = (F_1(x),\ F_2(x),\cdots ,\ F_m(x))^\top\\
\text{subject to} & x\in\RR^n,
\end{array}
\end{equation}
for a vector-valued function
$F:\RR^n\rightarrow (\RR\cup\{\infty\})^m$.
For each $i=1,\ldots,m$, we assume that the component
$F_i:\RR^n\rightarrow\RR$ 
is expressed as a composite form
\begin{equation}\label{comp-form}
F_i(x) := f_i(x) + g_i(x)
\end{equation}
where $f_i:\RR^n\to \RR$ is a continuously differentiable function and $g_i:\RR^n\to \RR\cup\{\infty\}$ is a proper, lower semi-continuous, and convex function.

\paragraph{Related works}

Over a couple of decades, \emph{descent methods} in multiobjective optimization have been developed as extensions of the ones from single objective case, which includes, for instance, the steepest descent method \cite{fliege2000steepest}, the projected gradient method \cite{drummond2004projected, fukuda2011convergence}, quasi-Newton method \cite{qu2011quasi}, etc.
Such methods aim to find a suitable search direction along which the next iterate is updated.
For multiobjective optimization with the composite form~\eqref{comp-form}, a proximal gradient method (PGM) was proposed by Tanabe et al. \cite{tanabe2019proximal}. Subsequently, the same authors \cite{tanabe2023convergence} provided convergence rate results of the PGM where the functions $f_i$ are assumed to have Lipschitz gradients. 
The convergence rate analyses were performed in terms of the \emph{merit functions} $u_0(x)$ and  $w_\ell(x)$ proposed in \cite{tanabe2024merit} that measure the concepts of weak Pareto optimality and Pareto stationarity at a point $x$, respectively.
Under Lipschitz continuity of $\nabla f_i$'s, the PGM proposed in \cite{tanabe2023convergence} guarantees to find an approximate Pareto stationary point $x$ satisfying $w_1(x)\leq \varepsilon$ within $O(\varepsilon^{-1})$ iterations. They also showed that, when all $f_i$ are convex, the same PGM finds an approximate weak Pareto optimal point $x$ satisfying $u_0(x)\leq \varepsilon$ in $O(\varepsilon^{-1})$ iterations. 

Beyond the Lipschitz continuity of gradients, parameter-free PGMs \cite{pinheiro2025universal,Amaral2025derivativefree} were recently developed for the problem \eqref{prob} when the components $f_i$ are \emph{weakly smooth}, i.e., the gradients $\nabla f_i$ are H\"older continuous. These methods are parameter-free in the sense that they do not need for knowing constants related to H\"{o}lder continuity such as the (worst) H\"older exponent $\nu_{\min} \in (0,1]$. In general non-convex setting, these PGMs guarantee a convergence rate of
$\displaystyle O(\varepsilon^{-\frac{1+\nu_{\min}}{2\nu_{\min} }} )$ to achieve $\ell_k^2\|d^k\|^2 \leq \varepsilon$ for a search direction $d^k$ and a linesearch parameter $\ell_k$ (which is fixed to $1$ in \cite{pinheiro2025universal}).
Although the convergence rates for weakly smooth and non-convex setting were analyzed by aforementioned studies,
the ones in weakly smooth and convex case have not been specifically addressed.

\paragraph{Contribution of this research}

In the present work, we propose a parameter-free multiobjective proximal gradient method (Algorithm~\ref{alg1}) by introducing new linesearch strategy that is adaptive to weak smoothness of the functions $f_i$. We establish analyses of iteration complexity of our method for both non-convex and convex cases.
The contribution of our work is summarized as follows (see also Table~\ref{table:works} in Section~\ref{sec:comparison} for a comparison with related works).
\begin{itemize}
\item For general non-convex case, it is shown in Theorem~\ref{Thm:noncon-comp} that our method finds an approximate Pareto stationary point $x^k$ satisfying $\ell_k w_{\ell_k}(x^k) \leq \varepsilon$ with iteration complexity at most $O(\varepsilon^{-\frac{1+\nu_{\min}}{2\nu_{\min}}})$ where $\ell_k \geq 1$ is a step parameter determined by linesearch.
Although the iteration complexity guarantee matches with the prior works \cite{Amaral2025derivativefree,pinheiro2025universal}, our criterion $\ell_k w_{\ell_k}(x^k) \leq \varepsilon$ achieves coercer approximation of Pareto stationarity because it implies both $w_1(x^k)\leq \varepsilon$ and $\ell_k^2\|d^k\|^2 \leq 2\varepsilon$ (as well as $\|d^k\|^2\leq 2\varepsilon$ with $\ell_k=1$).

\item In the case when all components $f_i$ of the multiobjective function are convex, we provide an enhanced convergence analysis improving the complexity bound compared to non-convex case.
The iteration complexity guarantee for achieving $\ell_k w_{\ell_k}(x^k)\leq \varepsilon$ is of $O(\varepsilon^{-\frac{1}{2\nu_{\min}}})$ as presented in Theorem~\ref{Thm:conv-comp}.
This is a novel result of the multiobjective PGM for the weakly smooth convex case.
Although a complexity analysis for another merit function $u_0(x^k)$ was conducted in previous research by Tanabe et al. \cite{tanabe2023convergence}, the complexity guarantee for $\ell_k w_{\ell_k}(x^k)$ is also important from practical point of view because it provides a verifiable stopping criterion.
\end{itemize}

The outline of this paper is as follows.
Section~\ref{sec:pre} presents several notations and concepts related to Pareto optimality and Pareto stationarity, and outlines the merit functions for multiobjective optimization proposed in \cite{tanabe2024merit}. Section~\ref{sec:proposed} presents the proposed PGM and establishes its iteration complexity guarantee in weakly smooth multiobjective optimization. Finally, Section~\ref{sec:conclusion} summarizes this paper.
\section{Preliminaries}\label{sec:pre}
This paper uses the symbol $[1:m] := \{1,2,\cdots, m\}$ for $m\in\NN$. We denote by $\top$ the transpose.

In this paper, we consider the multiobjective optimization problem \eqref{prob} in the composite form \eqref{comp-form}.
For the objective function $F:\RR^n\to (\RR\cup \{+\infty\})^m$ in the problem \eqref{prob}, its domain is defined by
$$\dom F := \{x\in \RR^n~|~F_i(x) < + \infty\ (\forall i\in[1:m])\},$$
for which we assume $\dom F \ne \emptyset$ throughout the paper. By the convexity of components $g_i$, the set $\dom F$ is convex.

In addition, throughout the paper, we assume that each component $f_i$ is \emph{weakly smooth} on $\dom F$ in the sense that the gradient $\nabla f_i$ is H\"{o}lder continuous on $\dom F$.
That is, assume that there exist a H\"older exponent $\nu_i\in (0,1]$ and a H\"older coefficient $M_i > 0$ such that
\begin{equation}\label{eq:Hol}
    \|\nabla f_i(x) - \nabla f_i(y)\|
\leq M_i\|x - y\|^{\nu_i},\quad \forall x,\ y\in \dom F.
\end{equation}
The case $\nu_i=1$ corresponds to $M_i$-Lipschitz continuity of $\nabla f_i$ that is useful for the convergence analysis of descent methods (e.g. \cite{tanabe2023convergence}). 
For weakly smooth functions $f_i$, the definition \eqref{eq:Hol} implies the following useful inequality:
$$
f_i(y) \leq f_i(x) + \nabla f_i(x)^\top (y-x) + \frac{M_i}{1+\nu_i} \|y-x\|^{1+\nu_i},\quad \forall x,y \in \dom F.
$$
In this paper, the following property is essential for the analysis of the proposed PGMs.

\begin{lemma}[{\cite[Lemma 2]{nesterov2015universal}}] \label{Lem:Nes}
Let $f:\RR^n\to\RR$ be a continuously differentiable function such that $\nabla f(x)$ is H\"older continuous on a convex set $D$ with an exponent $\nu\in(0,1]$ and a coefficient $M > 0$. Then, for any $\delta>0$ and $x,y \in D$, we have
\begin{equation}\label{eq:Nes-Holder}
f(y) \leq f(x) + \nabla f(x)^\top (y - x) +  \dfrac{1}{2} \left( \dfrac{1 - \nu}{1 + \nu}\cdot\dfrac{1}{\delta} \right)^\frac{1 - \nu}{1 + \nu}M^\frac{2}{1 + \nu} \|y - x\|^2 + \dfrac{\delta}{2}.
\end{equation}
\end{lemma}

\paragraph{Pareto optimality and merit functions}

Let us recall some concepts of optimality in multiobjective optimization problems. 
\begin{definition}
For the vector-valued function $F$ of problem \eqref{prob} and $x^\ast\in\dom F$, we say that
\begin{enumerate}[(i)]
\item
$x^\ast$ is a \emph{Pareto optimal point of $F$} if there is no $x \in \dom F$ that satisfies both $F(x) \leq F(x^\ast)$ and $F(x) \neq F(x^\ast)$;
\item
$x^\ast$ is a \emph{weakly Pareto optimal point of $F$} if there is no $x \in \dom F$ that satisfies $F(x) < F(x^\ast)$;
\item
$x^\ast$ is a \emph{Pareto stationary point of $F$} if $\displaystyle\max_{i\in[1:m]}F'(x^\ast;d)\geq 0$ holds for any direction $d\in\RR^n$.
\end{enumerate}
\end{definition}

Here, $F_i'(x^*;d):=\lim_{\alpha \downarrow 0} \frac{F_i(x^*+\alpha d)-F_i(x^*)}{\alpha} \left(= \nabla f_i(x^*)^\top d +g_i'(x^*;d)\right)$ represents the directional derivative of $F_i$ at $x^* \in \dom F$ along $d$, whose existence in $(-\infty,\infty]$ is ensured by the convexity of $g_i$. 
These optimality concepts are known to have the following relationships \cite{tanabe2019proximal}. 
\begin{lemma}
For the vector-valued function $F$ of problem \eqref{prob} and $x \in \dom F$, the following assertions hold.
\begin{enumerate}[(i)]
\item If $x$ is a Pareto optimal point of $F$, then $x$ is a weakly Pareto optimal point of $F$. 
\item If $x$ is a weakly Pareto optimal point of $F$, then $x$ is a Pareto stationary point of $F$.
\item If $x$ is a Pareto stationary point of $F$ and all components $F_i$ of $F$ are convex, then $x$ is a weakly Pareto optimal point of $F$. 
\end{enumerate}
\end{lemma}

Finally, we introduce \emph{merit functions} $w_\ell(x)$ and $u_0(x)$ that provide optimality measures of given points for the objective function $F$. 

Given a constant $\ell > 0$, we define the functions $w_\ell:\dom F\to \RR$ and $d_\ell:\dom F\to \RR^n$ by
\begin{align}
w_\ell(x) := &\max_{y\in \RR^n}\left[\min_{i\in[1:m]}\{\nabla f_i(x)^\top (x - y)\right.
\left. + g_i(x) - g_i(y)\} - \frac{\ell}{2}\|y - x\|^2\right],\\
d_\ell(x) := &\argmin_{d\in \RR^n}\left[ \max_{i\in[1:m]}\{\nabla f_i(x)^\top d + g_i(x + d) - g_i(x)\} + \dfrac{\ell}{2}\|d\|^2\right]. \label{compute:dl}
\end{align}
Clearly, the optimal value of \eqref{compute:dl} is given by $-w_\ell(x)$.
We also define the function $u_0:\dom F\to \RR\cup \{\infty\}$ as
$$
u_0(x) := \sup_{y\in\RR^n}\min_{i\in[1:m]}\left(F_i(x) - F_i(y) \right).
$$
These merit functions coincide with optimality measures commonly used in the single objective case $m=1$ and $g_1\equiv 0$ as they are given by $u_0(x) = F(x)-\inf F$, $w_\ell(x) = \frac{1}{2\ell}\|\nabla F(x)\|^2$, and $d_\ell(x) = -\frac{1}{\ell}\nabla F(x)$. 

These functions have been used in iteration complexity analysis of PGMs in the previous research \cite{tanabe2023convergence}, and the following facts are known.
\begin{lemma}[\cite{tanabe2024merit}]\label{Lem:merit}
Let $F$ be the vector-valued function of problem \eqref{prob}. Then, the following assertions hold for any $x\in \dom F$. 
\begin{enumerate}[(i)]
\item $w_\ell(x) \geq 0$ and
$u_0(x) \geq 0$ hold. 
\item $x$ is a Pareto stationary point of $F$ $\iff$ $w_\ell(x)=0$ $\iff$ $d_\ell(x)=0$.
\item $x$ is a weakly Pareto optimal point of $F$ if and only if $u_0(x)=0$.
\item For $\ell \geq r > 0$, 
$$w_\ell(x) \leq w_r(x) \leq \frac{\ell}{r}w_\ell(x)$$
hold. In particular, we have $\ell w_\ell(x) \geq w_1(x)$ for all $\ell \geq 1$.
\end{enumerate}
\end{lemma}

We also note the following relation between $w_\ell(x)$ and $d_\ell(x)$.

\begin{lemma}\label{lem:termination}
Given $x \in \dom F$ and $\ell>0$, we have
$$
\frac{\ell}{2}\|d_\ell(x)\|^2\leq w_{\ell}(x).
$$
\end{lemma}
\begin{proof}
Recall that $d_\ell(x)$ is the optimal solution to the following $\ell$-strongly convex minimization:
$$\min_{d}\left\{\max_{i\in[1:m]}(\nabla f(x)^\top d + g_i(x+d) - g_i(x)) + \frac{\ell}{2}\|d\|^2\right\}=-w_{\ell}(x).$$
Since the objective value is zero at $d=0$, the optimality of $d_\ell(x)$ and the $\ell$-strong convexity imply (see \cite[Theorem~5.25]{beck2017first})
$$
0 \geq -w_\ell(x) + \frac{\ell}{2}\|d_\ell(x)-0\|^2.
$$
\end{proof}

In this paper, the proposed method is designed to find approximate solutions satisfying $\ell w_\ell(x)\leq \varepsilon$, see Section~\ref{sec:proposed} for further discussions.  

\paragraph{Existing proximal gradient methods}

Here, we review some existing PGMs for (weakly) smooth multiobjective optimization problems \eqref{prob}.

A proximal gradient method for the problem \eqref{prob} were first developed by Tanabe et al. \cite{tanabe2019proximal} which extends earlier works on multiobjective steepest descent and projected gradient methods \cite{bonnel2005proximal,fliege2000steepest}.
Given an initial point $x^0 \in \dom F$, this method performs the iteration
\begin{equation}\label{tanabe-pgm}
x^{k+1} = x^k + d^k, \text{ where }
d^k = d_{\ell_k}(x^k) = \argmin_{d\in\RR^n} \left\{\max_{i\in [1:m]} \left(\nabla f_i(x^k)^\top d + g_i(x^k+d)-g_i(x^k)\right) + \frac{\ell_k}{2}\|d\|^2\right\},
\end{equation}
for a step parameter $\ell_k>0$.
If each $f_i$ has an $L$-Lipschitz continuous gradient, then the PGM~\eqref{tanabe-pgm} with a fixed step parameter $\ell_k\equiv \ell > L$ ensures a convergence rate
$
\min_{k \in [0:K-1]} w_1(x^k) \leq O(1/k)
$
in general and
$
u_0(x^k) \leq O(1/k)
$
in the case when $f_i$ are convex (see \cite{tanabe2023convergence}).

Regarding PGMs for weakly smooth objective functions, Pinheiro and Grapiglia \cite{pinheiro2025universal} proposed a projected gradient method when each $g_i$ is the indicator function of a closed convex feasible set $\Omega$.
This algorithm admits the search direction $d^k$ in \eqref{tanabe-pgm} with $\ell_k=1$ (see \cite[Remark 2]{pinheiro2025universal}).
They introduced a linesearch procedure to select a suitable step-size $\alpha_k>0$ incorporating the update $x^{k+1} = x^k + \alpha_k d^k$ instead of the one in \eqref{tanabe-pgm}, which enables us to adapt (unknown) parameters $M_i$ and $\nu_i$ in the H\"older continuity of $\nabla f_i$. Their algorithm ensures $\|d_{1}(x^k)\|^2\leq \varepsilon$ within $O(\varepsilon^{-\frac{1+\nu_{\min}}{2\nu_{\min}}})$ iterations where $\nu_{\min}=\min_{i\in[1:m]}\nu_i$ is the smallest H\"older exponent.

Another related method is a quasi-Newton type method for the problem \eqref{prob} proposed by Amaral et al. \cite{Amaral2025derivativefree}, which includes a PGM as its special case. In this method, the search direction $d^k$ is computed similar to \eqref{tanabe-pgm} with an additional quadratic term $d^\top B_{i,k} d$ for some (uniformly bounded) positive definite matrices $\{B_{i,k}\}$. To adapt the weak smoothness, this method employs a backtracking procedure to select suitable $\ell_k$ to ensure a sufficient decrease of the objective. Under the weak smoothness of each $f_i$, this method ensures $\ell_k^2\|d_{\ell_k}(x^k)\|^2\leq \varepsilon$ within $O(\varepsilon^{-\frac{1+\nu_{\min}}{2\nu_{\min}}})$ iterations \cite[Theorem 4.3]{Amaral2025derivativefree}.

\section{Proposed parameter-free proximal gradient method}\label{sec:proposed}
In this section, we propose a parameter-free PGM (Algorithm \ref{alg1}) for \eqref{prob} finding an approximate Pareto stationary point.
We will conduct its complexity analysis in terms of the merit function $\ell w_\ell(x)$ for the non-convex and the convex cases.
\begin{figure}[htbp]
\begin{algorithm}[H]
	\caption{Parameter-free proximal gradient method with backtracking}
	\label{alg1}
	\textbf{Input:} $\displaystyle x^0 \in \dom F$, $\varepsilon >0$, and $\ell_{-1}\geq 2$.
    \begin{algorithmic}[1]
    \FOR{$k = 0,1,2,\cdots$}
    \STATE Set $\ell \leftarrow \max\left\{\dfrac12,\dfrac{\ell_{k-1}}{4}\right\}.$
    \REPEAT
    \STATE Set $\ell \leftarrow 2\ell. $
    \STATE Compute $d^k \leftarrow d_\ell(x^k)=\underset{d\in \RR^n} {\operatorname{argmin}}\left[\displaystyle \max_{i\in[1:m]}\{\nabla f_i(x^k)^\top d + g_i(x^{k} + d) - g_i(x^k)\} + \dfrac{\ell}{2}\|d\|^2\right].$
    \STATE Compute $w_\ell(x^k) \leftarrow - \left[\displaystyle \max_{i\in[1:m]}\{\nabla f_i(x^k)^\top d^k + g_i(x^{k} + d^k) - g_i(x^k)\} + \dfrac{\ell}{2}\|d^k\|^2\right].$
    \STATE Set $x^{k+1} \leftarrow x^k + d^k.$
    \UNTIL{the following condition holds:
    \begin{equation}\label{eq:P} \tag{P}
        \forall i\in[1:m],\quad F_i(x^{k+1}) - F_i(x^k) \leq - w_\ell(x^k) + \dfrac{\varepsilon}{2\ell}.
    \end{equation}}
	\STATE Set $\ell_k \leftarrow \ell.$
    \ENDFOR
	\end{algorithmic}
\end{algorithm}
\end{figure}
The proposed method employs a backtracking strategy to find an appropriate parameter $\ell=\ell_k$ so that the descent condition \eqref{eq:P} is fulfilled.

In Algorithm~\ref{alg1}, we remark that the initial trial value of $\ell$ is actually $\max\left\{1,\dfrac{\ell_{k-1}}{2}\right\}$ in view of the steps 3 and 5.

Regarding the termination of proposed Algorithm~\ref{alg1}, it is reasonable to terminate the method when the merit function $w_\ell(x^k)$ is small enough achieving a given tolerance. In particular, when aiming to improve the accuracy of the approximate solution compared to {\rm\cite{Amaral2025derivativefree,pinheiro2025universal}}, we adopt the following termination condition:
$$\ell_k w_{\ell_k}(x^k) \leq \varepsilon.$$
It is important to note that $w_{\ell_k}(x^k)$ is computable provided the solvability of subproblem to obtain $d_{\ell_k}(x^k)$.
Since $\ell_k \geq 1$ by the construction, we have the following implications (by Lemma~\ref{Lem:merit} (iv) and Lemma \ref{lem:termination})
\begin{equation}
\begin{array}{ccc}\label{eq:termination}
\ell_kw_{\ell_k}(x^k) \leq \varepsilon &\Longrightarrow& w_{1}(x^k) \leq \varepsilon \\
\Downarrow& & \Downarrow\\
\ell_k^2\|d_{\ell_k}(x^k)\|^2\leq 2\varepsilon & & \|d_{1}(x^k)\|^2\leq 2\varepsilon.
\end{array}
\end{equation}
The conditions in \eqref{eq:termination} except $\ell_kw_{\ell_k}(x^k) \leq \varepsilon$ were used for convergence analysis of PGMs in \cite{Amaral2025derivativefree,pinheiro2025universal,tanabe2023convergence}.

The following lemma shows that the backtracking procedure in Algorithm~\ref{alg1} is well-defined and a bound of the number of inner loops can be obtained.
\begin{lemma}\label{Lem:noncon}
We have the following for Algorithm \ref{alg1}.
\begin{enumerate}[(i)]
\item At each iteration of Algorithm~\ref{alg1}, the condition \eqref{eq:P} is satisfied (i.e., the inner loop terminates) whenever
\begin{equation}\label{eq:ell-N}
\ell \geq \max_{i\in[1:m]}\left(\dfrac{1 - \nu_i}{1 + \nu_i}\dfrac{1}{\varepsilon}\right)^\frac{1 - \nu_i}{2 \nu_i} M_i^\frac{1}{\nu_i}=: N.
\end{equation}
\item The parameter $\ell_k$ admits the following bound.
$$\ell_k \leq \max\{\ell_{-1},2N\},\quad \forall k \geq -1.$$
\end{enumerate}
\end{lemma}
\begin{proof}
(i)
One can verify the equivalence
$$
\ell \geq \left(\dfrac{1 - \nu_i}{1 + \nu_i}\dfrac{1}{\varepsilon}\right)^\frac{1 - \nu_i}{2 \nu_i} M_i^\frac{1}{\nu_i}
\iff
\ell \geq N_i(\ell):=\left( \dfrac{1 - \nu_i}{1 + \nu_i}\cdot\dfrac{\ell}{\varepsilon} \right)^\frac{1 - \nu_i}{1 + \nu_i}M_i^\frac{2}{1 + \nu_i}.
$$
Therefore, the assumption $\ell \geq N:=\max_{i\in[1:m]}\left(\dfrac{1 - \nu_i}{1 + \nu_i}\dfrac{1}{\varepsilon}\right)^\frac{1 - \nu_i}{2 \nu_i} M_i^\frac{1}{\nu_i}$ given in \eqref{eq:ell-N} implies that $\ell \geq N_i(\ell)$ holds for all $i \in [1:m]$.
By applying Lemma \ref{Lem:Nes} with $\delta := \dfrac{\varepsilon}{\ell}$, we obtain
\begin{equation*}
\begin{split}
F_i(x^{k+1}) &= f_i(x^k + d^k) + g_i(x^k + d^k)\\
&\leq f_i(x^k) + \nabla f_i(x^k)^\top d^k + \dfrac{N_i(\ell)}{2}\|d^k\|^{2} + \dfrac{\delta}{2} + g_i(x^k + d^k)\\
&= F_i(x^k) + \nabla f_i(x^k)^\top d^k + g_i(x^k + d^k) - g_i(x^k) + \dfrac{N_i(\ell)}{2}\|d^k\|^2 + \dfrac{\varepsilon}{2\ell}\\
&\leq F_i(x^k) + \nabla f_i(x^k)^\top d^k + g_i(x^k + d^k) - g_i(x^k) + \dfrac{\ell}{2}\|d^k\|^2 + \dfrac{\varepsilon}{2\ell}, 
\end{split}
\end{equation*}
for all $k\in\NN \cup \{0\}\,$ and $i\in[1:m]$. Therefore, we conclude
\begin{equation*}
\begin{split}
F_i(x^{k+1}) - F_i(x^k)&\leq \nabla f_i(x^k)^\top d^k + g_i(x^k + d^k) - g_i(x^k) + \dfrac{\ell}{2}\|d^k\|^2 + \dfrac{\varepsilon}{2\ell}\\
&\leq \left\{ \max_{i\in[1:m]}(\nabla f_i(x^k)^\top d^k + g_i(x^k + d^k) - g_i(x^k)) + \dfrac{\ell}{2}\|d^k\|^2 \right\} + \dfrac{\varepsilon}{2\ell}\\
&= \min_{d\in \RR^n}\left\{ \max_{i\in[1:m]}(\nabla f_i(x^k)^\top d + g_i(x^k + d) - g_i(x^k)) + \dfrac{\ell}{2}\|d\|^2 \right\} + \dfrac{\varepsilon}{2\ell}\\
&= - w_{\ell}(x^k) + \dfrac{\varepsilon}{2\ell}.  
\end{split}
\end{equation*}
Thus, the condition \eqref{eq:P} is satisfied. 

(ii)
To prove the second assertion, it suffices to prove the inequality
\begin{equation}\label{eq:upper1}
\ell_k \leq \max\left\{\ell_{-1},\frac{\ell_{k-1}}2,2N\right\}, \quad \forall k \geq 0,
\end{equation}
because a recursive usage of \eqref{eq:upper1} implies
$$
\ell_k \leq \max\left\{\ell_{-1},\frac{\ell_{-1}}{2^{k+1}},2N\right\} = \max\left\{\ell_{-1},2N\right\}. 
$$
At for each iteration $k \geq 0$, we consider two cases based on whether a failure of verifying the condition \eqref{eq:P} occurred in backtracking procedure. 
When the condition \eqref{eq:P} is satisfied at the first trial of the backtracking, we have
$$
\ell_k = \max\left\{1,\dfrac{\ell_{k-1}}2\right\} \leq \max\left\{\ell_{-1},\dfrac{\ell_{k-1}}2\right\} \leq \max\left\{\ell_{-1},\dfrac{\ell_{k-1}}2, 2N\right\}.
$$
On the other case when the condition \eqref{eq:P} failed at least once, the second last trial $\ell=\ell_k/2$ does not satisfy \eqref{eq:P}. Then,
 we must have $\ell_k/2 < N$ by the assertion (i) and so
$$
\ell_k < 2N \leq \max\left\{\ell_{-1},\frac{\ell_{k-1}}2,2N\right\}.
$$
Hence, we conclude \eqref{eq:upper1} which implies the assertion (ii).
\end{proof}

Now, we establish a complexity bound for Algorithm~\ref{alg1} to find an $\varepsilon$-approximate solution in terms of the measure $\ell_k w_{\ell_k}(x^k)$.

\begin{theorem}\label{Thm:noncon-comp}
In the problem \eqref{prob}, let $\{x^k\}$ and $\{\ell_k\}$ be generated by Algorithm \ref{alg1} for a given $\varepsilon>0$. For any iteration number $K$ satisfying
\begin{equation*}
K \geq
\max\left\{
2\ell_{-1}\varepsilon^{-1},
\max_{i\in[1:m]} 4\left(\dfrac{1 - \nu_i}{1 + \nu_i}\right)^\frac{1 - \nu_i}{2 \nu_i} \varepsilon^{-\frac{1 + \nu_i}{2\nu_i}} M_i^\frac{1}{\nu_i} \right\}u_0(x^0) = O\left(\varepsilon^{-\frac{1+\nu_{\min}}{2\nu_{\min}}}\right),
\end{equation*}
where $\nu_{\min}=\min_{i\in[1:m]}\nu_i$, we have
\begin{equation}\label{eq:wl-eps}
    \min_{k\in[0:K-1]}\ell_kw_{\ell_k}(x^k) \leq \varepsilon.    
\end{equation}
\end{theorem}
\begin{proof}
We show the contraposition of the statement. 
Suppose that \eqref{eq:wl-eps} remains unfulfilled, i.e.,
\begin{equation}
\ell_kw_{\ell_k}(x^k) > \varepsilon,\quad k\in[0,K-1].
\end{equation}
Then, the condition \eqref{eq:P} implies
\begin{equation}\label{eq:obj-descent}
    F_i(x^{k+1}) - F_i(x^k) \leq - w_{\ell_k}(x^k) + \dfrac{\varepsilon}{2\ell_{k}} < - \dfrac{\varepsilon}{2\ell_{k}}
\end{equation}
for all $i\in[1:m]$ and $k\in[0,K-1]$. 
Adding both sides for $k=0,\ldots,K-1$, we get
$$F_i(x^K) - F_i(x^0) < - \dfrac{\varepsilon}{2} \sum_{k=0}^{K-1}\dfrac{1}{\ell_k}. $$
Now, using $\ell_k \leq \max\{\ell_{-1},2N\}$ from Lemma \ref{Lem:noncon}, we have
\begin{align*}
\dfrac{\varepsilon}{2}
&< \dfrac{F_i(x^0) - F_i(x^K)}{\displaystyle\sum_{k=0}^{K-1}\dfrac{1}{\ell_k}}
\leq \dfrac{\max\left\{\ell_{-1},2N\right\}\{F_i(x^0) - F_i(x^K)\}}K. 
\end{align*}
Since this inequality holds for any $i$, it follows that
\begin{align}
\dfrac{\varepsilon}{2}
&< \dfrac{\max\left\{\ell_{-1},2N\right\}\displaystyle\min_{j\in[1:m]}\{F_j(x^0) - F_j(x^K)\}}K\\
&\leq \dfrac{\max\left\{\ell_{-1},2N\right\}\displaystyle\sup_{x\in\RR^n}\min_{j\in[1:m]}\{F_j(x^0) - F_j(x)\}}K\\
&= \dfrac{\max\left\{\ell_{-1},2N\right\}u_0(x^0)}K. \label{eq:proof-nonconv-K-bound}
\end{align}
This can be arranged as follows.
\begin{equation}
K < \dfrac{2\cdot\max\{\ell_{-1},2N\}\displaystyle u_0(x^0)}{\varepsilon} =\max\left\{
2\ell_{-1}\varepsilon^{-1},
\max_{i\in[1:m]} 4\left(\dfrac{1 -  \nu_i}{1 + \nu_i}\right)^\frac{1 - \nu_i}{2 \nu_i} \varepsilon^{-\frac{1 + \nu_i}{2\nu_i}} M_i^\frac{1}{\nu_i} \right\}u_0(x^0).
\end{equation}
\end{proof}
\subsection{Complexity analysis in convex case} \label{sec:con}
Next, we show that the iteration complexity of the proposed PGM (Algorithm~\ref{alg1}) can be improved in the case where each component $f_i$ is a \emph{convex} function.

Before stating the main result, we prepare the following property for the convergence analysis.

\begin{lemma}\label{Lem:conv2}
Assume that each $f_i$ is convex in problem \eqref{prob}.
Then, the sequence $\{x^k\}$ generated by Algorithm~\ref{alg1} satisfies
\begin{equation}
\min_{i\in[1:m]}\left(F_i(x^{k+1}) - F_i(x)\right) \leq \dfrac{\ell_k}{2}\left(\|x^k - x\|^2 - \|x^{k+1} - x\|^2\right) + \dfrac{\varepsilon}{2\ell_k},\quad \forall x\in \dom F. 
\end{equation}
\end{lemma}
\begin{proof}
We define an $\ell_k$-strongly convex function by
$$
\phi_k(d):=\max_{i \in [1:m]} \left[ \nabla f_i(x^k)^\top d + g_i(x^k+d) - g_i(x^k) \right] + \frac{\ell_k}{2}\|d\|^2,
$$
whose minimizer is $d^k=x^{k+1}-x^k$ and whose minimum is given by $\min_{d\in\RR^n}\phi_k(d) = - w_{\ell_k}(x^k)$ in Algorithm \ref{alg1}.
For the generated sequence $\{x^k\}$ by Algorithm \ref{alg1}, the condition \eqref{eq:P} implies
\begin{equation*}
F_i(x^{k+1}) - F_i(x^k)\leq  - w_{\ell_k}(x^k) + \dfrac{\varepsilon}{2\ell_{k}} = \min_{d \in \RR^n} \phi_k(d) + \frac{\varepsilon}{2\ell_k}.
\end{equation*}
From this and the convexity of $f_i$, we have for all $x\in\dom F$ that 
\begin{equation*}
\begin{split}
F_i(x^{k+1}) - F_i(x) &= F_i(x^{k+1}) - F_i(x^k) + F_i(x^{k}) - F_i(x)\\
& \leq  \min_{d\in \RR^n}\phi_k(d)- \left[\nabla f_i(x^k)^\top (x - x^{k}) + g_i(x) - g_i(x^{k}) + \dfrac{\ell_k}{2}\|x - x^k\|^2\right] + \dfrac{\ell_k}{2}\|x - x^k\|^2 + \dfrac{\varepsilon}{2\ell_k}.
\end{split}
\end{equation*}
Taking $\min_{i \in [1:m]}$ on both sides, we have
\begin{align*}
\min_{i \in [1:m]}(F_i(x^{k+1}) - F_i(x))
 &\leq \min_{d\in \RR^n}\phi_k(d)-\phi_k(x-x^k) + \dfrac{\ell_k}{2}\|x - x^k\|^2 + \dfrac{\varepsilon}{2\ell_k}\\
 &\leq \left[\phi_k(x-x^k)-\frac{\ell_k}{2}\|(x-x^k)-d^k\|^2\right] -\phi_k(x-x^k)+ \dfrac{\ell_k}{2}\|x - x^k\|^2 + \dfrac{\varepsilon}{2\ell_k} \\
 &=\dfrac{\ell_k}{2}\|x - x^k\|^2 -\frac{\ell_k}{2}\|x-x^{k+1}\|^2 + \dfrac{\varepsilon}{2\ell_k},
\end{align*}
where the second inequality follows from the $\ell_k$-strong convexity of $\phi_k(\cdot)$ with its minimizer $d^k$ (see \cite[Theorem~5.25]{beck2017first}).
\end{proof}

As is imposed in the previous study \cite{tanabe2023convergence}, we assume the following conditions for the complexity analysis when each  $f_i$ is a convex function.

\begin{assumption}\label{Asm:conv-ana}
Let $X^\ast$ be the set of weakly Pareto optimal points of $F$, and for $\alpha\in \RR^m$, let \\
$\Omega_F(\alpha):= \{x\in \RR^n~|~F(x) \leq \alpha\}$ denote the level set.  We assume the following conditions.
\begin{enumerate}[(i)]
\item \label{Asm:2.1.1}
For all $x\in \Omega_F(F(x^0))$, there exists $x^\ast\in X^\ast$ such that $F(x^\ast) \leq F(x)$. 
\item \label{Asm:2.1.2}
$\displaystyle R := \sup_{F^\ast\in F(X^\ast \cap \Omega_F(F(x^0)))}\inf_{x\in F^{-1}(\{F^\ast\})}\|x - x^0\|^2 < \infty. $
\end{enumerate}
\end{assumption}

The constant $R$ is a generalization of the distance from $x^0$ to the optimal solution set $X^*$ in the single objective case ($m=1$).

Now we are ready to establish the main result for the convex case.

\begin{theorem}\label{Thm:conv-comp}
In the problem \eqref{prob}, assume that each $f_i$ is convex.
Let $\{x^k\}$ and $\{\ell_k\}$ be generated by Algorithm \ref{alg1} for a given $\varepsilon>0$.
Furthermore, let Assumption~\ref{Asm:conv-ana} hold. 
If the number of iterations $2K$ of Algorithm~\ref{alg1} satisfies
\[2K \geq  \max\left\{\ell_{-1},2N\right\}
\max\left\{4,\dfrac{2\sqrt{2R}}{\sqrt\varepsilon}\right\}= O\left(\varepsilon^{-\frac{1}{2\nu_{\min}}}\right),
\]
where $N$ is defined in \eqref{eq:ell-N} and $\nu_{\min}=\min_{i\in[1:m]}\nu_i$, then we have
\begin{equation}\label{eq:wl-eps-conv}
    \min_{k\in[0:2K-1]}\ell_kw_{\ell_k}(x^k) \leq \varepsilon.    
\end{equation}
\end{theorem}
\begin{proof}
We show the contraposition of the statement. 
Suppose that \eqref{eq:wl-eps-conv} remains unfulfilled, that is,
\begin{equation}
\ell_kw_{\ell_k}(x^k) > \varepsilon,\quad k\in[0,2K-1].
\end{equation}
Since the last iterate $x^{2K}$ is obtained by running $K$ iterations of Algorithm~\ref{alg1} with the initial point $x^K$ and the initial step parameter $\ell_{K-1}$, 
we have from (the contraposition of) Theorem~\ref{Thm:noncon-comp} that (see \eqref{eq:proof-nonconv-K-bound})
\begin{equation}\label{inequ:eps}
\frac{\varepsilon}{2} < \frac{\max\{\ell_{K-1},2N\}u_0(x^K)}{K} \leq \frac{\max\{\ell_{-1},2N\}u_0(x^K)}{K},
\end{equation}
where the last inequality follows by Lemma~\ref{Lem:noncon}(ii).

Next, we deal with an upper bound of $u_0(x^K)$.
Using the descent property \eqref{eq:obj-descent} and Lemma~\ref{Lem:conv2}, for all $k\in [0:K-1]$ and $x\in\dom F$, we have
\begin{equation*}
\min_{i\in[1:m]}\left(F_i(x^{K}) - F_i(x)\right)\leq
\min_{i\in[1:m]}\left(F_i(x^{k+1}) - F_i(x)\right) \leq \dfrac{\ell_k}{2}\left(\|x^k - x\|^2 - \|x^{k+1} - x\|^2\right) + \dfrac{\varepsilon}{2\ell_k}. 
\end{equation*}
Dividing by $\ell_k$ and adding both sides for $k \in [0:K-1]$, we obtain
\begin{equation*}
\begin{split}
\min_{i\in[1:m]}\left(F_i(x^{K}) - F_i(x)\right)\sum_{k=0}^{K-1}\dfrac{1}{\ell_k}
&\leq \dfrac{1}{2}\left(\|x^0 - x\|^2 - \|x^{K} - x\|^2\right) + \dfrac{\varepsilon}{2}\sum_{k=0}^{K-1}\dfrac{1}{\ell_k^2}\\
&\leq \dfrac{1}{2}\|x^0 - x\|^2 + \dfrac{\varepsilon}{2}\sum_{k=0}^{K-1}\dfrac{1}{\ell_k^2}.
\end{split}
\end{equation*}
Therefore, combining with Lemma~\ref{Lem:noncon}(ii), it follows that (recall $\ell_k\geq 1$)
\begin{equation*}
\begin{split}
\min_{i\in[1:m]}\left(F_i(x^{K}) - F_i(x)\right)
&\leq \dfrac{1}{2\displaystyle \sum_{k=0}^{K-1}\dfrac{1}{\ell_k}}\|x^0 - x\|^2 + \dfrac{\varepsilon}{2}
\cdot \dfrac{\displaystyle \sum_{k=0}^{K-1}\dfrac{1}{\ell_k^2}}{\displaystyle \sum_{k=0}^{K-1}\dfrac{1}{\ell_k}}
\leq \dfrac{\max\left\{\dfrac{\ell_{-1}}{2},N\right\}}{K}\|x^0 - x\|^2 + \dfrac{\varepsilon}{2}. 
\end{split}
\end{equation*}
Now, we take $\displaystyle \sup_{F^\ast\in F(X^\ast \cap \Omega_F(F(x^0)))}\inf_{x\in F^{-1}(\{F^\ast\})}$ on both sides.
Using the constant $R$ defined in Assumption~\ref{Asm:conv-ana}, we obtain
\begin{equation}\label{inequ:conv}
\sup_{F^\ast\in F(X^\ast \cap \Omega_F(F(x^0)))}\inf_{x\in F^{-1}(\{F^\ast\})}
\min_{i\in[1:m]}\left(F_i(x^{K}) - F_i(x)\right)
\leq \dfrac{\max\left\{\dfrac{\ell_{-1}}{2},N\right\}}{K}R + \dfrac{\varepsilon}{2}.
\end{equation}
We claim that the left hand side coincides with $u_0(x^K)$. In fact, we have
\begin{equation*}
\begin{split}
\sup_{F^\ast\in F(X^\ast \cap \Omega_F(F(x^0)))}\inf_{x\in F^{-1}(\{F^\ast\})}\min_{i\in[1:m]}\left(F_i(x^{K}) - F_i(x)\right)
&=\sup_{F^\ast\in F(X^\ast \cap \Omega_F(F(x^0)))}\min_{i\in[1:m]}\left(F_i(x^{K}) - F_i^*\right)\\
&= \sup_{x\in X^\ast \cap\Omega_F(F(x^0))}\min_{i\in[1:m]}\left(F_i(x^{K}) - F_i(x)\right)\\
&= \sup_{x\in \Omega_F(F(x^0))}\min_{i\in[1:m]}\left(F_i(x^{K}) - F_i(x)\right),
\end{split}
\end{equation*}
where the last equality follows from Assumption \ref{Asm:conv-ana}~(i). Since $\min_{i\in[1:m]}\left(F_i(x^{K}) - F_i(x)\right) \geq 0$ if and only if $x \in \Omega_F(F(x^K))$, and the descent property \eqref{eq:obj-descent} implies $\Omega_F(F(x^K))\subset \Omega_F(F(x^0))$, the last expression can be arranged as
\begin{align*}
\sup_{x\in \Omega_F(F(x^0))}\min_{i\in[1:m]}\left(F_i(x^{K}) - F_i(x)\right)
&= \sup_{x\in \Omega_F(F(x^K))}\min_{i\in[1:m]}\left(F_i(x^{K}) - F_i(x)\right)\\
&= \sup_{x\in \RR^n}\min_{i\in[1:m]}\left(F_i(x^{K}) - F_i(x)\right)\\
&= u_0(x^K).
\end{align*}
This verifies the claim and so the inequality \eqref{inequ:conv} becomes
$$u_0(x^K) \leq \dfrac{R\cdot\max\left\{\dfrac{\ell_{-1}}{2},N\right\}}{K} + \dfrac{\varepsilon}{2}.$$
Finally, incorporating this into the inequality \eqref{inequ:eps}, we conclude that
\begin{align}
\varepsilon 
&< \frac{2\max\{\ell_{-1},2N\}u_0(x^K)}{K} \leq \frac{R[\max\{\ell_{-1},2N\}]^2}{K^2} 
+ \frac{\varepsilon\max\{\ell_{-1},2N\}}{K}
\\
&\leq 2\max\{\ell_{-1},2N\}\cdot\max\left\{\frac{R\cdot\max\{\ell_{-1},2N\}}{K^2} 
, \frac{\varepsilon}{K}\right\}.
\end{align}
The number of iterations $2K$ satisfying this inequality has the following upper bound:
\begin{align*}
2K &< \max\left\{\ell_{-1},2N\right\}
\max\left\{4,\dfrac{2\sqrt{2R}}{\sqrt\varepsilon}\right\}= O(\varepsilon^{-\frac{1}{2\nu_{\min}}}).
\end{align*}
The proof of the theorem is completed.
\end{proof}

\subsection{Comparison with related methods}\label{sec:comparison}
We compare the iteration complexity results (Theorems~\ref{Thm:noncon-comp} and \ref{Thm:conv-comp}) of the proposed method with related ones \cite{tanabe2023convergence,pinheiro2025universal,Amaral2025derivativefree} as summarized in Table~\ref{table:works}. 
For the H\"older exponent $\nu_i$ of each $\nabla f_i$, we denote 
$\nu_{\min}:= \min_{i\in [1:m]}\nu_i$.
In both the non-convex and convex cases, when $f_i$ is weakly smooth, $\nu_{\min}$ affects the dominant factor in the iteration complexity of the proposed method. From this, it can be inferred that the iteration complexity of the proposed method is influenced by the component of the gradient of each $f_i$ that is furthest from Lipschitz continuity.

\begin{table}[htbp]
\begin{center}
\caption{Relation between the proposed method and previous studies on PGMs for multiobjective optimization.}\label{table:works}
\scalebox{0.9}{
\begin{tabular}{|r||l|l|l|p{7em}|p{5em}|} \hline 
  Algorithm &Convexity & Assumption of $\nabla f_i$& Optimality measure & Iteration complexity & Parameter-Free\\ \hline\hline
  Tanabe et al. \cite{tanabe2023convergence} & non-convex & Lipschitz continuity&$\displaystyle w_{1}(x^k) \leq \varepsilon $ & $O\left(\varepsilon^{-1}\right)$ & $\times$\\
  {}Pinheiro and Grapiglia  \cite{pinheiro2025universal} & non-convex & H\"older continuity& $\displaystyle \|d_{1}(x^k)\|^2 \leq \varepsilon$ & $O\left(\varepsilon^{-\frac{1 + \nu_{\operatorname{min}}}{2\nu_{\operatorname{min}}}}\right)$ & $\checkmark$\\ 
  {}Amaral et al. \cite{Amaral2025derivativefree} & non-convex & H\"older continuity& $\ell_k^2\|d_{\ell_k}(x^k)\|^2\leq \varepsilon$ & $O\left(\varepsilon^{-\frac{1 + \nu_{\operatorname{min}}}{2\nu_{\operatorname{min}}}}\right)$ & $\checkmark$\\ 
   Ours (Algorithm \ref{alg1}) & non-convex & H\"older continuity & $\displaystyle \ell_kw_{\ell_k}(x^k) \leq \varepsilon $ & $O\left(\varepsilon^{-\frac{1 + \nu_{\operatorname{min}}}{2\nu_{\operatorname{min}}}}\right)$ & $\checkmark$\\
{}Tanabe et al. \cite{tanabe2023convergence} & convex & Lipschitz continuity & $u_0(x^k) \leq \varepsilon$ & $O\left(\varepsilon^{-1}\right)$ & $\times$\\
   Ours (Algorithm \ref{alg1}) & convex  & H\"older continuity& $\ell_kw_{\ell_k}(x^k) \leq \varepsilon$ & 
   $O\left(\varepsilon^{-\frac{1}{2\nu_{\min}}}\right)$ & $\checkmark$\\ \hline
 \end{tabular}
}
\end{center}
\end{table}

Here, we give some discussions regarding the comparison in Table~\ref{table:works}.

\begin{itemize}
\item
The previous works \cite{pinheiro2025universal} and \cite{Amaral2025derivativefree} originally analyzed 
the iteration complexity to achieve $\|d_{1}(x^k)\|\leq \varepsilon$
and $\ell_k\|d_{\ell_k}(x^k)\|\leq \varepsilon$, respectively. 
For fair comparison with our method, Table~\ref{table:works} shows the iteration complexity using the measures $\|d_{1}(x^k)\|^2\leq \varepsilon$ 
and $\ell_k^2\|d_{\ell_k}(x^k)\|^2\leq \varepsilon$, 
in view of the implication from the condition 
$\ell_kw_{\ell_k}(x^k) \leq \varepsilon$ (see \eqref{eq:termination}).
\item As can be seen from Table~\ref{table:works}, the non-convex case achieved the same iteration complexity as previous studies \cite{pinheiro2025universal,Amaral2025derivativefree}.
This result, which guarantees $\ell_k w_{\ell_k}(x^k)\leq \varepsilon$ with an iteration complexity of $O\left(\varepsilon^{-\frac{1 + \nu_{\operatorname{min}}}{2\nu_{\operatorname{min}}}}\right)$, also encompasses the one $O(\varepsilon^{-1})$ shown by \cite{tanabe2023convergence} when each $\nabla f_i$ is Lipschitz continuous.
\item The iteration complexity shown in Table~\ref{table:works} for the convex case is a novel result of this study.
This result guarantees $\ell_k w_{\ell_k}(x^k) \leq \varepsilon$ with an iteration complexity of $O\left(\varepsilon^{-\frac{1}{2\nu_{\operatorname{min}}}}\right)$, which improves the one in the non-convex case. 
While Tanabe et al. \cite{tanabe2023convergence} provided a complexity guarantee for the merit function $u_0(x^k)$, this measure is hard to verify as a stopping criterion in general. In contrast, the measure $\ell_k w_{\ell_k}(x^k)$ is computable at each iteration, so providing its complexity analysis is important from practical aspect.
\end{itemize}

\section{Conclusion}\label{sec:conclusion}

In this paper, we demonstrated parameter-free proximal gradient methods for multiobjective optimization problems with weakly smooth components.
In the non-convex setting, an iteration complexity was obtained consistent to existing results and the convergence analysis is performed in terms of the criterion $\ell_k w_{\ell_k}(x^k) \leq \varepsilon$ which is coercer than the ones in previous works \cite{Amaral2025derivativefree,pinheiro2025universal,tanabe2023convergence}.
In the convex setting, we established a novel complexity result for weakly smooth problems, improving the complexity guarantee in non-convex case. 

Although the proposed method ensures to obtain an $\varepsilon$-approximate solution, the tolerance parameter $\varepsilon>0$ is fixed as an input of the algorithm, from which a global convergence property of the merit function is not guaranteed. This limitation is also seen in the work \cite{Amaral2025derivativefree}, and even in the work \cite{nesterov2015universal} in a single objective and convex case, while \cite{pinheiro2025universal} relaxes it for constrained problems. It would be valuable to develop PGMs with convergence guarantee of the merit function.

\subsection*{Acknowledgments}
The second author was supported partly by the JSPS KAKENHI Grant Number 25K15010.
The third author was supported partly by the JSPS KAKENHI Grant Numbers 25K21158 and 26K02871.

\subsection*{Data availability}
There is no data associated with this paper. 

\bibliography{reference.bib}

@book{beck2017first,
  title={First-order methods in optimization},
  author={Beck, Amir},
  year={2017},
  publisher={SIAM}
}

@article{evans1984overview,
  title={An overview of techniques for solving multiobjective mathematical programs},
  author={Evans, Gerald W},
  journal={Management Science},
  volume={30},
  number={11},
  pages={1268--1282},
  year={1984},
  publisher={INFORMS}
}

@article{fliege2000steepest,
  title={Steepest descent methods for multicriteria optimization},
  author={Fliege, J{\"o}rg and Svaiter, Benar Fux},
  journal={Mathematical methods of operations research},
  volume={51},
  number={3},
  pages={479--494},
  year={2000},
  publisher={Springer}
}

@article{marler2004survey,
  title={Survey of multi-objective optimization methods for engineering},
  author={Marler, R Timothy and Arora, Jasbir S},
  journal={Structural and multidisciplinary optimization},
  volume={26},
  number={6},
  pages={369--395},
  year={2004},
  publisher={Springer}
}

@article{drummond2004projected,
  title={A projected gradient method for vector optimization problems},
  author={Drummond, LM Grana and Iusem, Alfredo N},
  journal={Computational Optimization and applications},
  volume={28},
  number={1},
  pages={5--29},
  year={2004},
  publisher={Springer}
}

@article{bonnel2005proximal,
  title={Proximal methods in vector optimization},
  author={Bonnel, Henri and Iusem, Alfredo Noel and Svaiter, Benar Fux},
  journal={SIAM Journal on Optimization},
  volume={15},
  number={4},
  pages={953--970},
  year={2005},
  publisher={SIAM}
}

@article{qu2011quasi,
  title={Quasi-{N}ewton methods for solving multiobjective optimization},
  author={Qu, Shaojian and Goh, Mark and Chan, Felix TS},
  journal={Operations Research Letters},
  volume={39},
  number={5},
  pages={397--399},
  year={2011},
  publisher={Elsevier}
}

@article{fukuda2011convergence,
  title={On the convergence of the projected gradient method for vector optimization},
  author={Fukuda, Ellen H and Drummond, LM Grana},
  journal={Optimization},
  volume={60},
  number={8-9},
  pages={1009--1021},
  year={2011},
  publisher={Taylor \& Francis}
}

@article{fukuda2014survey,
  title={A survey on multiobjective descent methods},
  author={Fukuda, Ellen H and Drummond, Luis Mauricio Gra{\~n}a},
  journal={Pesquisa Operacional},
  volume={34},
  pages={585--620},
  year={2014},
  publisher={SciELO Brasil}
}

@article{nesterov2015universal,
  title={Universal gradient methods for convex optimization problems},
  author={Nesterov, Yu},
  journal={Mathematical Programming},
  volume={152},
  number={1},
  pages={381--404},
  year={2015},
  publisher={Springer}
}

@article{sener2018multi,
  title={Multi-task learning as multi-objective optimization},
  author={Sener, Ozan and Koltun, Vladlen},
  journal={Advances in neural information processing systems},
  volume={31},
  year={2018}
}

@article{tanabe2019proximal,
  title={Proximal gradient methods for multiobjective optimization and their applications},
  author={Tanabe, Hiroki and Fukuda, Ellen H and Yamashita, Nobuo},
  journal={Computational Optimization and Applications},
  volume={72},
  number={2},
  pages={339--361},
  year={2019},
  publisher={Springer}
}

@article{tanabe2023convergence,
  title={Convergence rates analysis of a multiobjective proximal gradient method},
  author={Tanabe, Hiroki and Fukuda, Ellen H and Yamashita, Nobuo},
  journal={Optimization Letters},
  volume={17},
  number={2},
  pages={333--350},
  year={2023},
  publisher={Springer}
}

@article{tanabe2024merit,
    title={New merit functions for multiobjective optimization and their properties},
    author={Tanabe, Hiroki and Fukuda, Ellen H and Yamashita, Nobuo},
    year={2024},
    journal={Optimization},
    volume={73},
    number={13},
    pages={3821--3858}
}

@article{pinheiro2025universal,
  title={Universal nonmonotone line search method for nonconvex multiobjective optimization problems with convex constraints},
  author={Pinheiro, Maria Eduarda and Grapiglia, Geovani Nunes},
  journal={Computational and Applied Mathematics},
  volume={44},
  number={2},
  pages={56},
  year={2025},
  publisher={Springer}
}

@article{Amaral2025derivativefree,
    author={Amaral, V. S. and Assun\c{c}\~{a}o, P. B. and Souza, D. R.},
    title={A Partially Derivative-Free Proximal Method for Composite Multiobjective Optimization in the {H}\"older Setting},
    journal = {arXiv preprint arXiv:2508.20071},
    year={2025}
}
\bibliographystyle{abbrvnat}

\end{document}